\documentclass{article}
\usepackage{amsmath}
\usepackage{amsfonts}
\usepackage{amssymb}
\providecommand{\U}[1]{\protect\rule{.1in}{.1in}}
\newtheorem{theorem}{Theorem}

\newtheorem{corollary}[theorem]{Corollary}

\newtheorem{definition}[theorem]{Definition}

\newtheorem{lemma}[theorem]{Lemma}
\newtheorem{proposition}[theorem]{Proposition}
\newtheorem{remark}[theorem]{Remark}

\newtheorem{hypothesis}{Hypothesis}
\newcommand{\loc}{\mathrm{loc}}
\newcommand{\eps}{\varepsilon}
\newcommand{\RR}{\mathbb{R}}
\newenvironment{proof}[1][Proof]{\noindent\textbf{#1.} }{\ \rule{0.5em}{0.5em}}
\begin{document}

% select a language among: english, french, italian
%
%\selectlanguage{english}

% if you dont have footnote, cancel \footnotemark[1]
% separate several authors with ' - '
%
 \title{Remarks on the stochastic transport equation with H\"{o}lder drift}
\author{F. Flandoli$^{1}$, M. Gubinelli$^{2}$, E. Priola$^{3}$\\{\small {(1) Dipartimento di Matematica, Universit\`{a} di Pisa, Italia}}\\{\small {\ (2) CEREMADE (UMR 7534), Universit\'e Paris Dauphine, France}}\\{\small {\ (3) Dipartimento di Matematica, Universit\`a di Torino, Italia }}}
\maketitle

\begin{abstract}
We consider a stochastic linear transport equation with a globally H\"{o}lder
continuous and bounded vector field. Opposite to what happens in the
deterministic case where shocks may appear, we show that the unique solution
starting with a $C^{1}$-initial condition remains of class $C^{1}$ in space.
We also improve some results of \cite{FGP} about well-posedness. Moreover, we
prove a stability property for the solution with respect to the initial datum.

\end{abstract}

\section{Introduction\label{section 1}}

The aim of this paper is twofold. On one side, we review ideas and recent
results about the regularization by noise in ODEs and PDEs (Section
\ref{section 1}). On the other, we give detailed proof of two new results of
regularization by noise, for linear trasport equations, related to those of
the paper \cite{FGP} (Theorem \ref{thmnoblowup} and the results of section
\ref{sec:fractional}).

\subsection{The ODE case}

A well known but still always surprising fact is the regularization produced
by noise on ordinary differential equations (ODEs). Consider the ODE in
$\mathbb{R}^{d}$
\[
\frac{d}{dt}X\left(  t\right)  =b\left(  t,X\left(  t\right)  \right)  ,\quad
X\left(  0\right)  =x_{0}\in\mathbb{R}^{d}%
\]
with $b:\left[  0,T\right]
\times\mathbb{R}^{d}\rightarrow\mathbb{R}^{d}$. If $b$ is Lipschitz
continuous and has linear growth, uniformly in $t$, then there
exists a unique solution $X\in C\left(  [  0,T]
;\mathbb{R}^{d}\right)  $. But when $b$ is less regular there are
well-known counterexamples, like the case $d=1$, $b\left(  x\right)
=2sign\left( x\right)  \sqrt{\left\vert x\right\vert }$, $x_{0}=0$
where the Cauchy problem has infinitely many solutions: $X\left(
t\right)  =0$, $X\left(  t\right) =t^{2}$, $\displaystyle{ X\left(  t\right)
=-t^{2}}$, and others. The function $b$ of this example is H\"{o}lder
continuous.

Consider now the stochastic differential equation (SDE)%
\begin{equation}
\label{SDE}dX\left(  t\right)  =b\left(  t,X\left(  t\right)  \right)
dt+\sigma dW\left(  t\right)  ,\quad X\left(  0\right)  =x_{0}\in
\mathbb{R}^{d}%
\end{equation}
with $\sigma\in\mathbb{R}$ and $\left\{  W\left(  t\right)  \right\}
_{t\geq0}$ a $d$-dimensional Brownian motion on a probability space
$\left( \Omega,{\mathcal F},P\right)  $. We say that a continuous
stochastic process $X\left( t,\omega\right)  $, $t\geq0$,
$\omega\in\Omega$, adapted to the filtration $\{ {\mathcal
F}_{t}^{W}\}_{t\geq 0}$ of the Brownian motion, is a solution if it
satisfies the identity
\[
X\left(  t,\omega\right)  =x_{0}+\int_{0}^{t}b\left( s, X\left(
s,\omega \right)  \right)  ds+ \sigma W\left(  t,\omega\right)
,\quad t\geq 0,
\]
for $P$-a.e. $\omega\in\Omega$. In the Lipschitz case we have again
existence and uniqueness of solutions. But now, we have more: if
$\sigma\neq0$ and $b\in
L^{\infty}\left(  \left[  0,T\right]  \times\mathbb{R}^{d};\mathbb{R}%
^{d}\right)  $ then there is existence and uniqueness of solutions, \cite{V}.
The result is true even when $b\in L^{q}\left(  0,T;L^{p}\left(
\mathbb{R}^{d};\mathbb{R}^{d}\right)  \right)  $ with $\frac{d}{p}+\frac{2}%
{q}<1$, $p,q\geq2$ \cite{Kry-Ro} (the assumptions can be properly localized).
Recently, we have proved in \cite{FGP}
the following additional result, which
 will be used below (the function spaces are defined in
 Section 1.4).

\begin{theorem}
If $\sigma\neq0$ and $b\in L^{\infty}\left(  0,T ;C_{b}^{\alpha}\left(
\mathbb{R}^{d};\mathbb{R}^{d}\right)  \right)  $,
 $\alpha \in (0,1)$, then there exists a
stochastic flow of diffeomorphisms $\phi_{t} = \phi\left(  t,\omega\right)  $
associated to the SDE, with $D\phi\left(  t,\omega\right)  $ and $D\phi
^{-1}\left(  t,\omega\right)  $ of class $C^{\alpha^{\prime}}$ for every
$\alpha^{\prime}\in\left(  0,\alpha\right)  $.
\end{theorem}

By stochastic flow of diffeomorphisms we mean a family of maps $\phi\left(
t,\omega\right)  :\mathbb{R}^{d}\rightarrow\mathbb{R}^{d}$ such that:
\begin{itemize}
\item[i)] $\phi\left(  t,\omega\right)  \left(  x_{0}\right)  $ is the
 unique
solution of the SDE for every $x_{0}\in\mathbb{R}^{d}$;
\item[ii)]\ $\phi\left(  t,\omega\right)  $ is a diffeomorphisms of $\mathbb{R}^{d}$.
\end{itemize}
For several results on stochastic flows under more regular
conditions on $b$ see \cite{K}. Let us give an idea of the proof
assuming $\sigma =1$. Introduce the vector valued non homogeneous
backward parabolic equation%
\begin{align*}
\frac{\partial U}{\partial t}+b\cdot\nabla U+\frac{1}{2}\Delta U  &
=-b+\lambda U\qquad\text{on }\left[  0,T\right] \\
U\left(  T,x\right)   &  =0
\end{align*}
with $\lambda\geq0$. By parabolic regularity theory we have the
following result (cf.  Theorem 2 in \cite{FGP}):
\begin{theorem}
\label{theorem SDE flux}If $b\in L^{\infty}\left(  0,T
;C_{b}^{\alpha}\left( \mathbb{R}^{d};\mathbb{R}^{d}\right)  \right)
$, $\alpha \in (0,1)$, then there exists a unique bounded and
locally Lipschitz
 solution $U$ with the property
\[
\frac{\partial U}{\partial t} \in L^{\infty}\left( 0,T
;C_{b}^{\alpha}( \mathbb{R}^{d}; \mathbb{R}^{d}) \right),\quad
D^{2}U\in L^{\infty}\left( 0,T ;C_{b}^{\alpha}\left( \mathbb{R}^{d};
\mathbb{R}^{d} \otimes \mathbb{R}^{d} \otimes \mathbb{R}^{d} \right)
\right).
\]
Moreover, for large $\lambda$ one has, for any $(t,x) \in [0,T]
\times \mathbb R^d$,
\[
\left\vert \nabla U\left(  t,x\right)  \right\vert \leq\frac{1}{2}.
\]
\end{theorem}
 \noindent If $X\left(  t\right)  $ is a solution of the SDE, we apply It\^{o}
formula to
$U\left(  t,X\left(  t\right)  \right)  $ and get%
\[
U\left(  t,X\left(  t\right)  \right)  =U\left(  0,x_{0}\right)  +\int_{0}%
^{t}\mathcal{L}U\left(  s,X\left(  s\right)  \right)  ds+\int_{0}^{t}\nabla
U\left(  s,X\left(  s\right)  \right)  dW\left(  s\right)
\]
where $\mathcal{L}U=\frac{\partial U}{\partial t}+b\cdot\nabla U+\frac{1}%
{2}\Delta U$. Hence, being $\mathcal{L}U=-b+\lambda U$,
\[
U\left(  t,X\left(  t\right)  \right)  =U\left(  0,x_{0}\right)  +\int_{0}%
^{t}\left(  -b+\lambda U\right)  \left(  s,X\left(  s\right)  \right)
ds+\int_{0}^{t}\nabla U\left(  s,X\left(  s\right)  \right)  dW\left(
s\right)
\]
and thus%
\begin{align*}
\int_{0}^{t}b\left(s,  X\left(  s\right)  \right)  ds  &  =U\left(
0,x_{0}\right)  -U\left(  t,X\left(  t\right)  \right)
+\int_{0}^{t}\lambda
U\left(  s,X\left(  s\right)  \right)  ds\\
&  +\int_{0}^{t}\nabla U\left(  s,X\left(  s\right)  \right)  dW\left(
s\right)  .
\end{align*}
In other words, we may rewrite the SDE as%
\begin{align*}
X\left(  t\right)   &  =x_{0}+U\left(  0,x_{0}\right)  -U\left(  t,X\left(
t\right)  \right)  +\int_{0}^{t}\lambda U\left(  s,X\left(  s\right)  \right)
ds\\
&  +\int_{0}^{t}\nabla U\left(  s,X\left(  s\right)  \right)  dW\left(
s\right)  +W\left(  t\right)  .
\end{align*}
The advantage is that $U$ is twice more regular than $b$ and $\nabla U$ is
once more regular. All terms in this equation are at least Lipschitz continuous.

From the new equation satisfied by $X\left(  t\right)  $ it is easy to prove
uniqueness, for instance. But, arguing a little bit formally,
 it is also clear
that we have differentiability of $X\left(  t\right)  $ with respect to the
initial condition $x_{0}$. Indeed, if $D_{h}X\left(  t\right)  $ denotes the
derivative in the direction $h$, we (formally)\ have%
\begin{align*}
D_{h}X\left(  t\right)   &  =h+D_{h}U\left(  0,x_{0}\right) - \nabla
U\left(
t,X\left(  t\right)  \right)  D_{h}X\left(  t\right) \\
&  +\int_{0}^{t}\lambda \nabla U\left(  s,X\left(  s\right)  \right)
 D_{h}X\left(
s\right)  ds\\
&  +\int_{0}^{t}D^{2}U\left(  s,X\left(  s\right)  \right)  D_{h}X\left(
s\right)  dW\left(  s\right)  .
\end{align*}
All terms are meaningful (for instance the tensor valued coefficient
$D^{2}U\left(  s,X\left(  s\right)  \right)  $ is bounded
continuous), $\nabla U\left(  t,X\left(  t\right)  \right)  $ has
norm less than 1/2 (hence the term  $\nabla U\left(  t,X\left(
t\right) \right) D_{h}X\left(  t\right)  $ contracts) and one can
prove that this equation has a solution $D_{h}X\left( s\right)  $.
Along these lines one can build a rigorous proof of
differentiability. We do not discuss the other properties.

%\vskip 2mm

\begin{remark}
\emph{A main open problem is the case when $b$ is random:%
 $
 b=b\left(  \omega,t,x\right)
 $.
In this case, strong uniqueness statements of the previous form are unknown
(when $b$ is not regular). }
\end{remark}

\subsection{The PDE case}

We have seen that noise improves the theory of ODEs. Is it the same for PDEs?
We have several more possibilities, several dichotomies:%
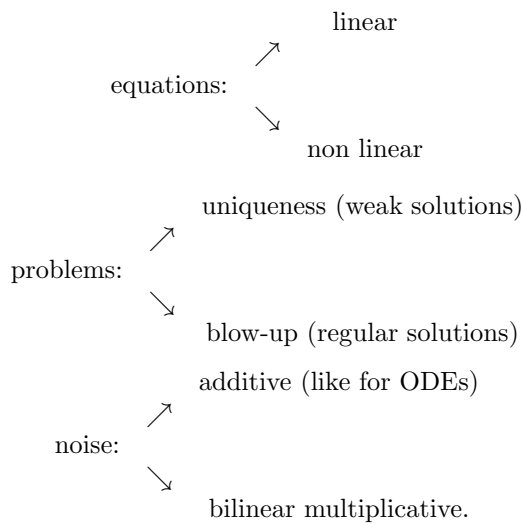
\[%
\begin{array}
[c]{ccc}
&  & \text{linear}\\
& \nearrow & \\
\text{equations:} &  & \\
& \searrow & \\
&  & \text{non linear}%
\end{array}
\]%
\[%
\begin{array}
[c]{ccc}
&  & \text{uniqueness (weak solutions)}\\
& \nearrow & \\
\text{problems:} &  & \\
& \searrow & \\
&  & \text{blow-up (regular solutions)}%
\end{array}
\]%
\[%
\begin{array}
[c]{ccc}
&  & \text{additive (like for ODEs)}\\
& \nearrow & \\
\text{noise:} &  & \\
& \searrow & \\
&  & \text{bilinear multiplicative.}%
\end{array}
\]
Let us deal with two of the simplest but not trivial combinations:
\textit{linear} transport equations, both the problem of
\textit{uniqueness}
of weak $L^{\infty}$ solutions and of \textit{no blow-up} of $C^{1}%
$-solutions, the improvements of the deterministic theory produced by a
\textit{bilinear multiplicative} noise.

The linear deterministic transport equation is the first order PDE in
$\mathbb{R}^{d}$%
\[
\frac{\partial u}{\partial t}+b\cdot\nabla u=0,\qquad u|_{t=0}=u_{0}%
\]
where $b:\left[  0,T\right]  \times\mathbb{R}^{d}\rightarrow\mathbb{R}^{d}$ is
given and we look for a solution $u:\left[  0,T\right]  \times\mathbb{R}%
^{d}\rightarrow\mathbb{R}$.

\begin{definition}
Assume $b,\operatorname{div}b\in L_{loc}^{1}
 = L_{loc}^{1} ([0,T] \times {\mathbb {R}}^d) $,
 $u_{0}\in L^{\infty}\left(
\mathbb{R}^{d}\right)  $. We say that $u$ is a weak
$L^{\infty}$-solution if:
\begin{itemize}
\item[i)] $u\in L^{\infty}\left(  \left[  0,T\right]  \times\mathbb{R}^{d}\right)  $
\item[ii)] for all $\theta\in C_{0}^{\infty}\left(  \mathbb{R}^{d}\right)  $ one has%
\[
\int_{{\mathbb R}^d}u\left(  t,x\right)  \theta\left(  x\right)
dx=\int_{{\mathbb R}^d}u_{0}\left( x\right)  \theta\left(  x\right)
dx+\int_{0}^{t}\int_{{\mathbb R}^d}u\left(  s,x\right)
\operatorname{div}\left( b\left(s,  x\right)  \theta\left(  x\right)
\right) dxds
\]
\end{itemize}
\end{definition}
Existence of weak $L^{\infty}$-solutions is a general fact, obtained
by weak-star compactness methods. When $b\in L^{\infty}\left(  0,T ;
Lip_b\left(  \mathbb{R}^{d};\mathbb{R}^{d}\right)  \right)  $,
uniqueness can be proved, and also existence of smoother solutions
when $u_{0}$ is smoother. Moreover, one has the transport relation
\[
u\left(  t,\phi\left(  t,x\right)  \right)  =u_{0}\left(  x\right)
\]
where $\phi\left(  t,x\right)  $ is the deterministic flow associated to the
equation of characteristics%
\[
\frac{d}{dt}\phi\left(  t,x\right)  =
 b\left(  \phi\left(  t,x\right)  \right)
,\quad\phi\left(  0,x\right)  =x.
\]
When $b$
is less than Lipschitz continuous, there are counterexamples. For
instance, for
\[
d=1,\quad b\left(  x\right)  =2sign\left(  x\right)  \sqrt{\left\vert
x\right\vert }%
\]
the PDE has infinitely many solutions from any initial condition $u_{0}$.
These solutions coincide for $\left\vert x\right\vert >t^{2}$, where the flow
is uniquely defined, but they can be prolonged almost arbitrarily for
$\left\vert x\right\vert <t^{2}$, for instance setting%
\[
u\left(  t,x\right)  =C\text{ for }\left\vert x\right\vert <t^{2}%
\]
with arbitrary $C$.
Remarkable is the result of \cite{DiPernaLions} which states that the solution
is unique when (we do not stress the generality of the behavior at infinity)%
\begin{equation}
\nabla b\in L_{loc}^{1}\left(  \left[  0,T\right]  \times\mathbb{R}^{d}%
;\mathbb{R}^{d}\right),  \label{weak diff}
\end{equation}%
\begin{equation}
\operatorname{div}b\in L^{1}\left(  0,T;L^{\infty}\left(  \mathbb{R}%
^{d},\mathbb{R}^{d}\right)  \right)  .\label{no compression}%
\end{equation}
There are generalizations of this result (for instance \cite{Ambrosio}), but
not so far from it. In these cases the flow exists and is unique but only in a
proper generalized sense. The assumption (\ref{no compression}) is the
quantitative one used to prove the estimate (for simplicity we omit the cut-of
needed to localize)%
\[
\int_{{\mathbb R}^d}u^{2}\left(  t,x\right)  dx= \int_{{\mathbb
R}^d}u_{0}^{2}\left(  x\right)  dx+\int _{0}^{t}ds \int_{{\mathbb
R}^d}u^{2}\left(  s,x\right) \operatorname{div}b\left(s,  x\right)
dx
\]%
\[
\leq\int_{{\mathbb R}^d}u_{0}^{2}\left(  x\right)
dx+\int_{0}^{t}\left\Vert \operatorname{div}b\left(  s, \cdot
\right) \right\Vert _{\infty}ds \int_{{\mathbb R}^d}u^{2}\left(
s,x\right)  dx
\]
which implies, by Gronwall lemma, $\int_{{\mathbb R}^d}u^{2}\left(
t,x\right)  dx=0$ when $u_{0}=0$ (this implies uniqueness, since the
equation is linear). The assumption (\ref{weak diff}) apparently has
no role but it is essential to perform these computations
rigorously. One has to prove that a weak $L^{\infty}$-solution $u$
satisfies the previous identity. In order to apply differential
calculus to $u$, one can mollify $u$ but then a remainder, a
commutator, appears in the equation. The convergence to zero of this
commutator (established by the so called \textit{commutator lemma}
of \cite{DiPernaLions}) requires assumption (\ref{weak diff}). We
have recalled these facts since they are a main motiv below.

The problem of no blow-up of $C^{1}$ or $W^{1,p}$ solution is open for the
deterministic equation, under essentially weaker conditions than Lipschitz
continuity of $b$. The equation satisfied by first derivatives $v_{k}%
=\frac{\partial u}{\partial x_{k}}$ involves derivatives of $b$ as a potential
term%
\[
\frac{\partial v_{k}}{\partial t}+b\cdot\nabla v_{k}+\sum_{i}\frac{\partial
b}{\partial x_{i}}v_{i}=0,\qquad v_{k}|_{t=0}=\frac{\partial u_{0}}{\partial
x_{k}}%
\]
and $L^{\infty}$ bounds on $\frac{\partial b}{\partial x_{i}}$ seem necessary
to control $v_{k}$. Again there are simple counterexamples: in the case
\[
d=1,\quad b\left(  x\right)  =-2sign\left(  x\right)  \sqrt{\left\vert
x\right\vert },%
\]
the equation of characteristics has coalescing trajectories (the
solutions from $\pm x_{0}$ meet at $x=0$ at time $\sqrt{\left\vert
x_{0}\right\vert }$) and thus, if we start with a smooth initial
condition $u_{0}$ such that at some point $x_{0}$ satisfies
$u_{0}\left(  x_{0}\right)  \neq u_{0}\left( -x_{0}\right)  $, then
at time $t_{0}=\sqrt{\left\vert x_{0}\right\vert }$ the solution is
discontinuous (unless $u_{0}$ is special, the discontinuity appears
immediately, for $t>0$).

Consider the following stochastic version of the linear transport equation:%
\[
\frac{\partial u}{\partial t}+b\cdot\nabla u+\sigma\nabla u\circ\frac{dW}%
{dt}=0,\qquad u|_{t=0}=u_{0}.
\]
The noise $W$ is a $d$-dimensional Brownian motion, $\sigma\in\mathbb{R}$, the
operation $\nabla u\circ\frac{dW}{dt}$ has simultaneously two features:\ it is
a scalar product between the vectors $\nabla u$ and $\frac{dW}{dt}$, and has
to be interpreted in the Stratonovich sense. The noise has a transport
structure as the deterministic part of the equation. It is like to add the
fast oscillating term $\sigma\frac{dW}{dt}$ to the drift $b$:%
\[
b\left(  x\right)  \longrightarrow b\left(  x\right)  +\sigma\frac{dW}%
{dt}\left(  t\right)  .
\]

Concerning Stratonovich calculus and its relation with It\^{o} calculus, see
\cite{K}. We recall the so called Wong-Zakai principle (proved as a rigorous
theorem in several cases): when one takes a differential equations with a
smooth approximation of Brownian motion, and then takes the limit towards true
Brownian motion, the correct limit equation involves Stratonovich integrals.
Thus equations with Stratonovich integrals are more physically based.

\begin{definition}
Assume $b,\operatorname{div}b\in L_{loc}^{1}$, $u_{0}\in
L^{\infty}\left( \mathbb{R}^{d}\right)  $. We say that a stochastic
process $u$ is a weak $L^{\infty}$-solution of the SPDE if:
\begin{itemize}
\item[i)] $u\in L^{\infty}\left(  \Omega\times\left[  0,T\right]  \times
\mathbb{R}^{d}\right)  $
\item[ii)] for all $\theta\in C_{0}^{\infty}\left(  \mathbb{R}^{d}\right)  $, $\int_{{\mathbb R}^d}u\left(  t,x\right)  \theta\left(  x\right)  dx$ is a continuous adapted semimartingale
\item[iii)] for all $\theta\in C_{0}^{\infty}\left(  \mathbb{R}^{d}\right)  $, one
has%
\begin{align*}
\int_{{\mathbb R}^d}u\left(  t,x\right)  \theta\left(  x\right)  dx
&  =\int_{{\mathbb R}^d}u_{0}\left( x\right)  \theta\left(  x\right)
dx+\int_{0}^{t}\int_{{\mathbb R}^d}u\left(  s,x\right)
\operatorname{div}\left( b\left(s,  x\right)  \theta\left(  x\right)
\right)
dxds\\
&  +\sigma\int_{0}^{t}\left(  \int_{{\mathbb R}^d}u\left(
s,x\right)  \nabla\theta\left( x\right)  dx\right)  \circ dW\left(
s\right)  .
\end{align*}
\end{itemize}
\end{definition}
The following theorem is due to \cite{FGP}.
\begin{theorem} \label{fgpp}
If $\sigma\neq0$ and%
 \begin{align} \label{c34}
b\in L^{\infty}\left(  0,T ;C_{b}^{\alpha}\left(  \mathbb{R}^{d}%
;\mathbb{R}^{d}\right)   \right)  ,\quad\operatorname{div}b\in L^{p}
([0,T] \times \mathbb{R}^d),
\end{align}
for some $\alpha\in\left(  0,1\right)  $ and $p>d\wedge2$, then
there exists a unique weak $L^{\infty}$-solution of the SPDE. If
$\alpha\in(1/2,1)$ then we have uniqueness only assuming
$\operatorname{div} b \in L^{1}_{loc}$.
Moreover, it holds%
\[
u\left(  t,\phi\left(  t,x\right)  \right)  =u_{0}\left(  x\right)
\]
where $\phi\left(  t,x\right)  $ is the stochastic flow of diffeomorphisms
associated to the equation%
\[
d\phi\left(  t,x\right)  =b\left( t, \phi\left(  t,x\right)  \right)
dt+\sigma dW\left(  t\right)  ,\quad\phi\left(  0,x\right)  =x
\]
given by Theorem \ref{theorem SDE flux}.
\end{theorem}

Thus we see that a suitable noise improves the theory of linear transport
equation from the view-point of uniqueness of weak solutions. One of the aims
of this paper is to prove a variant of this theorem, under different
assumptions on $b$. It requires a new form of commutator lemma with respect to
those proved in \cite{DiPernaLions} or \cite{FGP}.

Let us come to the blow-up problem.
%As a consequence of the results of
%\cite{FGP}, we prove below in Section \ref{section no %blow-up} the following result:
The following result can be deduced from  \cite[Appendix A]{FGP} in which we
have considered $BV_{loc}$-solutions  for the transport equation. In Section 2
we will give a direct proof of  the existence part which is of independent interest.

\begin{theorem}
\label{thmnoblowup}If $\sigma\neq0$,%
\[
b\in L^{\infty}(  0,T ;C_{b}^{\alpha}(  \mathbb{R}^{d}%
;\mathbb{R}^{d})  ) ,
\]
for some $\alpha\in\left(  0,1\right)  $ and $u_{0}\in C^{1}_{b}\left(
\mathbb{R}^{d}\right)  $, then there exists a unique classical $C^{1}%
$-solution for the transport equation with probability one. It is given by
\begin{align}
\label{df1}u\left(  t,x\right)  =u_{0}\left(  \phi_{t}^{-1}\left(  x\right)
\right)
\end{align}
where $\phi_{t}^{-1}$ is the inverse of the stochastic flow $\phi_{t}%
=\phi\left(  t,\cdot\right)  $.
\end{theorem}

The main claim of this theorem is the regularity of the solution for positive
times, which is new with respect to the deterministic case. The uniqueness
claim is known, as a particular case of a result in $BV_{loc}$, see Appendix 1
of [8].

Notice that, for solutions with such degree of regularity ($BV_{loc}$ or
$C^{1}$), no assumption on $\operatorname{div}b$ is required;
$\operatorname{div}b$ does not even appear in the definition of solution (see
below). On the contrary, to reach uniqueness in the much wider class of weak
$L^{\infty}$-solutions, in [8] we had to impose the additional condition (4)
on $\operatorname{div}b$, for some $p>d\wedge2$ ($\operatorname{div}b$ also
appears in the definition of weak $L^{\infty}$-solution); this happens also in
the deterministic theory.

\subsection{Some other works on regularization by noise}

The following list does not aim to be exhaustive, see for instance \cite{Fla} for other
results and references:

\begin{itemize}
\item the uniqueness for linear transport equations can be extended to other
weak assumptions on the drift, \cite{AF}, \cite{Mau1}; also no blow-up holds
for $L^{p}$ drift see \cite{FeFla} and  \cite{pro};

\item similar results hold for linear continuity equations, \cite{FGP cont},
\cite{Mau2}:
\[
\frac{\partial\rho}{\partial t}+\operatorname{div}\left(  b\rho\right)
=0,\qquad\rho|_{t=0}=\rho_{0}:
\]
a noise of the form $\nabla\rho\circ\frac{dW}{dt}$ prevents mass concentration;

\item analog results hold for the vector valued linear equations%
\[
\frac{\partial M}{\partial t}+\operatorname{curl}\left(  b\times M\right)  =0
\]
similar to the vorticity formulation of 3D Euler equations or
magneto-hydrodynamics, where the singularities in the deterministic case are
not shocks but infinite values of $M$; a noise of the form
$\operatorname{curl}\left(  e\times M\right)  \circ\frac{dW}{dt}$ prevents
blow-up \cite{F Nek};

\item improved Strichartz estimates for a special Schr\"{o}dinger model with
noise%
\[
i\partial_{t}u+\Delta u\circ\frac{dW}{dt}=0
\]
have been proved, which are stronger than the corresponding ones for
$i\partial_{t}u+\Delta u=0$ and allow to prevent blow-up in a non-linear case
when blow-up is possible without noise, see \cite{DT10};

\item nonlinear transport type equations of two forms have been investigated:
2D Euler equations and 1D Vlasov-Poisson equations; in these cases non-collapse
of measure valued solutions concentrated in a finite number of points has been
proved, \cite{FGP Euler}, \cite{DFV}.
\end{itemize}
We conclude the introduction with some notations.

\subsection {Notations} \label{sect notations} Usually we
denote by $D_{i}f$ the derivative in the $i$-th coordinate direction
and with $(e_{i})_{i=1,\dots,d}$ the canonical basis of
$\mathbb{R}^{d}$ so that $D_{i}f=e_{i}\cdot Df$. For partial
derivatives of any order $n\geq1$ we use the notation $D_{i_{1},...,i_{n}}%
^{n}$. If $\eta:\mathbb{R}^{d}\rightarrow\mathbb{R}^{d}$ is a $C^{1}%
$-diffeomorphism we will denote by $J\eta(x)=\text{det}[D\eta(x)]$ its
Jacobian determinant. For a given function $f$ depending on $t\in\lbrack0,T]$
and $x\in{\mathbb{R}}^{d}$, we will also adopt the notation $f_{t}(x)=f(t,x)$.

Let $T>0$ be fixed. For ${\alpha}\in(0,1)$ define the space
$L^{\infty}\left(
0,T;C_{b}^{\alpha}({\mathbb{R}}^{d})\right)  $ as the set of all bounded Borel
functions $f:[0,T]\times{\mathbb{R}}^{d}\rightarrow{\mathbb{R}}$ for which
\[
\lbrack f]_{\alpha,T}=\sup_{t\in\lbrack0,T]}\sup_{x\neq y\in{\mathbb{R}}^{d}%
}\frac{|f(t,x)-f(t,y)|}{|x-y|^{\alpha}}<\infty
\]
($|\cdot|$ denotes the Euclidean norm in ${\mathbb{R}}^{d}$ for every $d$, if
no confusion may arise). This is a Banach space with respect to the usual norm
$\Vert f\Vert_{{\alpha},T}=\Vert f\Vert_{0}+[f]_{\alpha,T}$ where $\Vert
f\Vert_{0}=\sup_{(t,x)\in\lbrack0,T]\times{\mathbb{R}}^{d}}|f(t,x)| $.
 Similarly, when $\alpha =1$ we define $L^{\infty}\left(
0,T; Lip_b({\mathbb{R}}^{d})\right)$.

We write $L^{\infty}\left(  0,T;C_{b}^{\alpha}({\mathbb{R}}^{d};{\mathbb{R}%
}^{d})\right)  $ for the space of all vector fields $f:[0,T]\times{\mathbb{R}%
}^{d}\rightarrow{\mathbb{R}}^{d}$ having all components in $L^{\infty}\left(
0,T;C_{b}^{\alpha}({\mathbb{R}}^{d})\right)  $.

%\vskip 1mm

Moreover, for $n\geq1,$ $f\in L^{\infty}\left(  0,T;C_{b}^{n+\alpha
}({\mathbb{R}}^{d})\right)  $ if all spatial partial derivatives
$D_{i_{1},...,i_{k}}^{k}f\in L^{\infty}\left(  0,T;C_{b}^{\alpha}({\mathbb{R}%
}^{d})\right)  $, for all orders $k=0,1,\dots,n$. Define the corresponding
norm as
\[
\Vert f\Vert_{n+\alpha,T}=\Vert f\Vert_{0}+\sum_{k=1}^{n}\Vert
D^{k}f\Vert _{0}+[D^{n}f]_{\alpha,T},
\]
where we extend the previous notations $\Vert\cdot\Vert_{0}$ and
$[\cdot]_{\alpha,T}$ to tensors. The definition of the space $L^{\infty
}\left(  0,T;C_{b}^{n+\alpha}({\mathbb{R}}^{d};{\mathbb{R}}^{d})\right)  $ is
similar.
%The previous functions spaces can be defined similarly when
%$T=+\infty$ (i.e., we are considering functions defined on $[0,\infty
%)\times\mathbb{R}^{d}$).
 The spaces $C_{b}^{n+{\alpha}}({\mathbb{R}}^{d})$ and
$C_{b}^{n+{\alpha}}({\mathbb{R}}^{d};{\mathbb{R}}^{d})$ are defined as before
but only involve functions $f:{\mathbb{R}}^{d}\rightarrow{\mathbb{R}}^{d}$
which do not depend on time. Moreover, we say that $f:\mathbb{R}%
^{d}\rightarrow\mathbb{R}^{d}$ belongs to $C^{n,\alpha}$, $n\in\mathbb{N} $,
$\alpha\in(0,1)$, if $f$ is continuous on $\mathbb{R}^{d}$, $n$-times
differentiable with all continuous derivatives and the derivatives of order
$n$ are locally $\alpha$-H\"{o}lder continuous. Finally, $C_{0}^{0}%
(\mathbb{R}^{d})$ denotes the space of all real continuous functions defined
on $\mathbb{R}^{d}$, having compact support and by $C_{0}^{\infty}%
(\mathbb{R}^{d})$ its subspace consisting of infinitely differentiable functions.

For any $r>0$ we denote by $B(r)$ the Euclidean ball centered in 0 of radius
$r$ and by $C_{r}^{\infty}(\mathbb{R}^{d})$ the space of smooth functions with
compact support in $B(r)$; moreover, $\Vert\cdot\Vert_{L_{r}^{p}}$ and
$\Vert\cdot\Vert_{W_{r}^{1,p}}$ stand for, respectively, the $L^{p}$-norm and
the $W^{1,p}$-norm on $B\left(  r\right)  $, $p\in\left[  1,\infty\right]  $.
We let also $[f]_{C^{\theta}_{r}}=\sup_{x\neq y \in B(r)}%
|f(x)-f(y)|/|x-y|^{\theta}$.

We will often use the standard mollifiers\textit{.} Let $\vartheta
:{\mathbb{R}}^{d}\rightarrow{\mathbb{R}}$ be a smooth test function such that
$0\leq\vartheta(x)\leq1$, $x\in{\mathbb{R}}^{d}$, $\vartheta(x)=\vartheta
(-x)$, $\int_{{\mathbb{R}}^{d}}\vartheta(x)dx=1$, $\mathrm{supp}%
\,(\vartheta)\subset$ $B(2)$, $\vartheta(x)=1$ when $x\in B(1)$. For any
${\varepsilon}>0$, let $\vartheta_{{\varepsilon}}(x)={\varepsilon}%
^{-d}\vartheta(x/{\varepsilon})$ and for any distribution $g:{\mathbb{R}}%
^{d}\rightarrow{\mathbb{R}}^{n}$ we define the mollified approximation
$g^{\varepsilon}$ as
\begin{equation}
g^{\varepsilon}(x)=\vartheta_{{\varepsilon}}\ast g(x)=g(\vartheta
_{{\varepsilon}}(x-\cdot)),\;\;\;x\in{\mathbb{R}}^{d}.\label{molli}%
\end{equation}
If $g$ depends also on time $t$, we consider $g^{\varepsilon}(t,x)=(\vartheta
_{{\varepsilon}}\ast g(t,\cdot))(x)$, $t\in\lbrack0,T]$, $x\in\mathbb{R}^{d}$.

%%%%%%%%%%%%%%%%%%%%%%%%%%%%%%%%%
\medskip Recall that, for any smooth bounded domain $\mathcal{D}$ of
$\mathbb{R}^{d}$, we have: $f \in W^{\theta, p} (\mathcal{D})$, $\theta
\in(0,1)$, $p \ge1$, if and only if $f \in L^{p} (\mathcal{D})$ and
\[
[f]_{W^{\theta, p} }^{p} =\iint_{\mathcal{D }\times\mathcal{D}} \frac{|f(x) -
f(y)|^{p}}{|x-y|^{\theta p + d}} dx dy < \infty.
\]
We have $W^{1, p} (\mathcal{D}) \subset W^{\theta, p} (\mathcal{D})$,
$\theta\in(0,1)$.
%%%%%%%%%%%%%%%%%%%%%%%%%%%%%%%%%

\bigskip In the sequel we will assume a stochastic basis with a
$d$-dimensional Brownian motion $\left(  \Omega,\left(  \mathcal{F}{}%
_{t}\right)  ,{}\mathcal{F},P,\left(  W_{t}\right)  \right)  $ to be given. We
denote by $\mathcal{F}_{s,t}$ the completed $\sigma$-algebra generated by
$W_{u}-W_{r}$, $s\leq r\leq u\leq t$, for each $0\le s<t$.

\smallskip

Let us finally recall our basic assumption on the drift vector
field.
\begin{hypothesis}
\label{hy1} There exists ${\alpha}\in(0,1)$ such that $b\in
L^{\infty}\left(
0,T;C_{b}^{\alpha}({\mathbb{R}}^{d};{\mathbb{R}}^{d})\right)  $.
\end{hypothesis}

\section{No blow-up in $C^{1}\label{section no blow-up}$}

This section is devoted to prove Theorem \ref{thmnoblowup}.
Since the solution
claimed by this theorem is regular, we do not need to integrate over test
functions in the term $b\cdot\nabla u$ and thus we do not need to require
$\operatorname{div}b\in L_{loc}^{1}$. For this reason, we modify the
definition of solution.
\begin{definition}
\label{c11} Assume $b\in L_{loc}^{1}$, $u_{0}\in C^{1}_{b}\left(
\mathbb{R}^{d}\right)  $. We say that a stochastic process $u  \in L^{\infty
}(\Omega\times\lbrack0,T]\times\mathbb{R}^{d})$ is a classical $C^{1}%
$-solution of the stochastic transport equation if:
\begin{itemize}
\item[i)] $u(\omega,t,\cdot)\in C^{1}(\mathbb{R}^{d})$ for a.e. $(\omega,t)\in
\Omega\times\lbrack0,T]$;
\item[ii)] for all $\theta\in C_{0}^{\infty}\left(  \mathbb{R}^{d}\right)  $, $\int_{{\mathbb R}^d}u\left(  t,x\right)  \theta\left(  x\right)  dx$ is a continuous adapted semimartingale;
\item[iii)] for all $\theta\in C_{0}^{\infty}\left(  \mathbb{R}^{d}\right)  $, one
has%
\begin{align*}
\int_{{\mathbb R}^d}u\left(  t,x\right)  \theta\left(  x\right)  dx
&  =\int_{{\mathbb R}^d}u_{0}\left( x\right)  \theta\left(  x\right)
dx-\int_{0}^{t}\int_{{\mathbb R}^d}b\left(  s,x\right)
\cdot\nabla u\left(  s,x\right)  \theta\left(  x\right)  dxds\\
&  +\sigma\int_{0}^{t}\left(  \int_{{\mathbb R}^d}u\left(
s,x\right)  \nabla\theta\left( x\right)  dx\right)  \circ dW\left(
s\right)  .
\end{align*}
\end{itemize}
\end{definition}
If $u$ is a classical $C^{1}$-solution and $\mathrm{div}\,b\in L_{loc}^{1}
\left(  \left[  0,T\right]  \times\mathbb{R}^{d}\right)  $, then $u$ is also a
weak $L^{\infty}$-solution.  Conversely, if $u$ is a weak $L^{\infty}%
$-solution, $u_{0}\in C^{1}_{b}\left(  \mathbb{R}^{d}\right)  $ and (i) is
satisfied then $u$ is a classical $C^{1}$-solution.

Before giving the proof we mention the following useful result proved in
\cite[Theorem 5]{FGP}:
\begin{theorem}
\label{th:flow1} Assume that Hypothesis \ref{hy1} holds true for some
$\alpha\in(0,1)$. Then we have the following facts:
\begin{itemize}
\item[(i)] (pathwise uniqueness) For every $s\in\left[  0,T\right]  $,
$x\in{\mathbb{R}}^{d}$, the stochastic equation (\ref{SDE}) has a unique
continuous adapted solution $X^{s,x}=\left(  X_{t}^{s,x}\big(\omega\right)
,t\in\left[  s,T\right]  ,$ $\omega\in\Omega\big)$.
\item[(ii)] (differentiable flow) There exists a stochastic flow $\phi_{s,t}$
of diffeomorphisms for equation (\ref{SDE}). The flow is also of class
$C^{1+{\alpha}^{\prime}}$ for any ${\alpha}^{\prime}<{\alpha}$.
\item[(iii)] (stability) Let $(b^{n})\subset L^{\infty}\left(  0,T;C_{b}%
^{\alpha}({\mathbb{R}}^{d};{\mathbb{R}}^{d})\right)  $ be a sequence of vector
fields and $\phi^{n}$ be the corresponding stochastic flows. If $b^{n}%
\rightarrow b$ in $L^{\infty}( 0,T;C_{b}^{\alpha^{\prime}}({\mathbb{R}}%
^{d};{\mathbb{R}}^{d})) $ for some $\alpha^{\prime}>0$, then, for any $p\geq
1$,
\begin{equation}
\lim_{n\rightarrow\infty}\sup_{x\in{\mathbb{R}}^{d}}\sup_{0\leq s\leq T}%
E[\sup_{r\in\lbrack s,T]}|\phi_{s,r}^{n}(x)-\phi_{s,r}(x)|^{p}%
]=0\label{stability1}%
\end{equation}
\begin{equation}
\sup_{n\in\mathbb{N}}\sup_{x\in{\mathbb{R}}^{d}}\sup_{0\leq s\leq T}%
E[\sup_{u\in\lbrack s,T]}\Vert D\phi_{s,u}^{n}(x)\Vert^{p}]<\infty
,\label{bound}%
\end{equation}
\begin{equation}
\lim_{n\rightarrow\infty}\sup_{x\in{\mathbb{R}}^{d}}\sup_{0\leq s\leq T}%
E[\sup_{r\in\lbrack s,T]}\Vert D\phi_{s,r}^{n}(x)-D\phi_{s,r}(x)\Vert
^{p}]=0.\label{stability2}%
\end{equation}
\end{itemize}
\end{theorem}
\begin{remark}
\label{inverse} \emph{We point out that the previous assertions
\eqref{stability1}, \eqref{bound} and \eqref{stability2} also holds when
$\phi_{s,r}^{n}(x)$ and $\phi_{s,r}(x)$ are replaced respectively by
$(\phi_{s,r}^{n})^{-1}(x)$ and $(\phi_{s,r})^{-1}(x)$. }

\emph{To see this note that for a fixed $t>0$, $Z_{s} = (\phi_{s,t})^{-1}(x)
$, $s \in[0,t]$, is measurable with respect to  $\mathcal{F}_{s,t}$ (the
completed $\sigma$-algebra generated by $W_{u}-W_{r}$, $s\leq r\leq u\leq t$,
for each $0\le s<t$) and solves
\begin{equation}
\label{neweq}Z_{s}=x-\int_{s}^{t}b(r, Z_{r})dr - \sigma[W_{t}-W_{s}].
\end{equation}
This is a simple backward stochastic differential equations, of the same form
as the original one (only the drift has opposite sign).  Note that for
regular
functions $f \in C^{2}_{b}(\mathbb{R}^{d})$, It\^o's formula becomes
\[
f(Z_{s}) = f(x) - \int_{s}^{t} \nabla f(Z_{r})
\cdot b(r, Z_{r})dr  - \int%
_{s}^{t} \nabla f (Z_{r}) \cdot dW_{r} -  \frac{\sigma^{2}}{2}
\int_{s}^{t} \triangle f ( Z_{r})dr
\]
where $\int_{s}^{t} \nabla f (Z_{r}) \cdot dW_{r}$ is the so called
backward It\^o integral (is a limit in probability of elementary
integrals like $\sum_{k} \nabla f (Z_{s_{k}}) \cdot(W_{s_{k}}-
W_{s_{k-1}})$ in which we consider the partition $s_{0}=0 < \ldots<
s_{N} =t$). Since this stochastic integral enjoys usual properties
of the classical It\^o  integral, one can repeat all the arguments
needed to prove \eqref{stability1}, \eqref{bound} and
\eqref{stability2} even for solutions $Z$ to \eqref{neweq}. }
\end{remark}

\begin{proof}
(\textbf{Theorem \ref{thmnoblowup}})
Under the assumptions of the theorem, it has been proved in Appendix 1 of [8]
that unqueness holds in $BV_{loc}$. Hence it holds in $C^{1}$. For this
result, no assumption on $\operatorname{div}b$ is required.

 We show now that \eqref{df1} is a classical
$C^{1}$-solution.
It is easy to check (i) in Definition \ref{c11}.  Moreover, if $\theta\in
C_{0}^{\infty}(\mathbb{R} ^{d})$, by changing variable we have:
%from
%Proposition~\ref{lemma:diff-lp},%
\[
\int_{{}{\mathbb{R}}^{d}}u(t,x)\theta(x)dx=\int_{{}{\mathbb{R}}^{d}}
u_{0}(y)\theta(\phi_{t}(y))J \phi_{t}(y)dy,
\]
where $J \phi_{t}(y) = \text{det}[D \phi_{t} (y) ]$,  and so also property
(ii) follows.
To prove property (iii) consider the flow $\phi_{t}^{\varepsilon}$ for the
regularized vector field $b^{\varepsilon}$ (see \eqref{molli}) and let $J
\phi_{t}^{\varepsilon}(y)$ be its Jacobian determinant. Note that $u_{0}
\circ(\phi^{\varepsilon}_{t})^{-1} \to u_{0} \circ\phi_{t}^{-1}$ weakly in
$L^{\infty}(\mathbb{R}^{d})$, uniformly in $t \in[0,T]$ and $P$-a.s., indeed
for $\theta\in C_{0}^{\infty}(\mathbb{R} ^{d})$ we have
\[
\int_{{}{\mathbb{R}}^{d}} (u_{0} \circ(\phi^{\varepsilon}_{t})^{-1} )(y)
\theta(y) dy = \int_{{}{\mathbb{R}}^{d}} u_{0}(y) \theta(\phi^{\varepsilon
}_{t}(y)) J \phi_{t}^{\varepsilon}(y) dy
\]%
\[
\to
\int_{{}{\mathbb{R}}^{d}} u_{0}(y) \theta(\phi_{t}(y)) J \phi_{t}(y) dy,
\]
 as $\epsilon \to 0$,
  using the properties of the stochastic flow stated in Theorem
 \ref{th:flow1}.
  By density we can extend this convergence to any $\theta\in
L^{1}(\mathbb{R} ^{d})$. Moreover since $b^{\varepsilon}$ is smooth,
it is easy to prove that
\[
dJ_{t}^{\varepsilon}(y)={\mathrm{div}\,}b^{\varepsilon}_{t}(\phi
_{t}^{\varepsilon}(y))J \phi_{t}^{\varepsilon}(y)dt
\]
and by the It\^o formula we find
\begin{equation}
\label{ii}
\begin{split}
\int_{{\mathbb{R}}^{d}}u_{0}(y) & \theta(\phi_{t}^{\varepsilon}(y))J_{t} %
^{\varepsilon}(y)dy  =
\\ &  \int_{{\mathbb{R}}^{d}}u_{0}(y)\theta(y)dy+\int%
_{0}^{t}ds \int_{{\mathbb{R}}^{d}}u_{0}(y)L^{b^{\varepsilon}}\theta(\phi
_{s}^{\varepsilon}(y))J \phi_{s}^{\varepsilon}(y)dy\\
&  +\int_{0}^{t}ds\int_{{\mathbb{R}}^{d}}u_{0}(y) \theta(\phi_{s}%
^{\varepsilon}(y)){\mathrm{div}\, }b^{\varepsilon}_{s} (\phi_{s}^{\varepsilon
}(y))J \phi_{s}^{\varepsilon}(y)dy\\
&  +\sigma\int_{0}^{t}dW_{s}\cdot\int_{{\mathbb{R}}^{d}}u_{0}(y)\nabla
\theta(\phi_{s}^{\varepsilon}(y))J\phi_{s}^{\varepsilon}(y)dy,
\end{split}
\end{equation}
where
\[
L^{b^{\varepsilon}} \theta(y) = \frac{1}{2}\sigma^{2} \Delta\theta
(y)+b^{\epsilon}_{s}(y)\cdot\nabla\theta(y).
\]
Note that, integrating by parts,
\[
\int_{0}^{t}ds\int_{{\mathbb{R}}^{d}}u_{0}(y) \theta(\phi_{s}^{\varepsilon
}(y)){\mathrm{div}\, }b^{\varepsilon}_{s} (\phi_{s}^{\varepsilon}(y))J
\phi_{s}^{\varepsilon}(y)dy
\]
\[
= \int_{0}^{t}ds\int_{{\mathbb{R}}^{d}}u_{0}((\phi_{s}^{\varepsilon})^{-1}(x)
) \theta(x){\mathrm{div}\, }b^{\varepsilon}_{s} (x) dx
\]
\[
= - \int_{0}^{t}ds\int_{{\mathbb{R}}^{d}}u_{0} ((\phi_{s}^{\varepsilon}%
)^{-1}(x) ) \nabla\theta(x)\cdot b^{\varepsilon}_{s} (x) dx
\]
\[
- \int_{0}^{t}ds\int_{{\mathbb{R}}^{d}} \nabla u_{0}((\phi_{s}^{\varepsilon
})^{-1}(x) ) D (\phi_{s}^{\varepsilon})^{-1}(x) \cdot b^{\varepsilon}_{s} (x)
\, \theta(x) dx.
\]
Therefore
\[
%\label{ii}
\begin{split}
\int_{{\mathbb{R}}^{d}}u_{0}(y) & \theta(\phi_{t}^{\varepsilon}(y))J_{t}%
^{\varepsilon}(y)dy = \\
& \int_{{\mathbb{R}}^{d}}u_{ 0}(y )\theta(y)dy+ \frac{1}{2}\sigma^{2}\int%
_{0}^{t}ds \int_{{\mathbb{R} }^{d}}u_{0}(( \phi_{s}^{\varepsilon})^{-1}(x) )
\triangle\theta(x) dx\\
& - \int_{0}^{t}ds\int_{{\mathbb{R}}^{d}} \nabla u_{0}((\phi_{s}^{\varepsilon
})^{-1}(x) ) D (\phi_{s}^{\varepsilon})^{-1}(x) \cdot b^{\varepsilon}_{s} (x)
\, \theta(x) dx\\
& +\sigma\int_{0}^{t}dW_{s}\cdot\int_{{\mathbb{R}}^{d}}u_{0}(y)\nabla\theta
(\phi_{s}^{\varepsilon}(y))J\phi_{s}^{\varepsilon}(y)dy.
\end{split}
\]
By changing variable $y = ( \phi_{s}^{\varepsilon})^{-1}(x)$ of the second and
third integral in the right-hand side, there are no problems to pass to the
limit as $\epsilon\to0$, $\mathbb{P}$-a.s., using (iii) in Theorem
\ref{th:flow1} and Remark  \ref{inverse} (precisely, one can pass to the limit
along a suitable sequence $(\epsilon_{n}) \subset(0,1)$ converging to 0).
To this purpose we only note that for the stochastic integral we have
\[
\int_{0}^{t}dW_{s}\cdot\int_{{\mathbb{R}}^{d}}u_{0}(y)\nabla\theta(\phi
_{s}^{\varepsilon}(y))J \phi_{s}^{\varepsilon}(y)dy \, \to\,  \int_{0}%
^{t}dW_{s}\cdot\int_{{\mathbb{R}}^{d}}u_{0}(y)\nabla\theta(\phi_{s}(y))J
\phi_{s}(y)dy
\]
uniformly on $[0,T]$ in $L^{2} (\Omega)$ as $\varepsilon\to0$.
Finally we get
\[
%\label{ii}%
\begin{split}
\int_{{\mathbb{R}}^{d}}u_{0}( ( \phi_{t})^{-1}(x) )\theta(x)dx =\int%
_{{\mathbb{R}}^{d}}u_{ 0}(y )\theta(y)dy+
\frac{\sigma^2}{2}\int_{0}^{t}ds
\int_{{\mathbb{R} }^{d}}u_{0}(( \phi_{s})^{-1}(x) ) \triangle\theta(x) dx\\
- \int_{0}^{t}ds\int_{{\mathbb{R}}^{d}} \nabla u_{0}((\phi_{s})^{-1}(x) ) D
(\phi_{s})^{-1}(x) \cdot b_{s} (x) \, \theta(x) dx\\
+\sigma\int_{0}^{t}dW_{s}\cdot\int_{{\mathbb{R}}^{d}}u_{0}(( \phi_{s}%
)^{-1}(x))\nabla\theta(x)dx.
\end{split}
\]
By passing from It\^o to Stratonovich integral this is exactly the formula we
wanted to prove. The proof is complete.
\end{proof}

\begin{remark}
\emph{One can show that the boundedness assumption  on $b$ is not important to
prove the previous Theorem \ref{thmnoblowup}. Indeed at least when $b $ is
independent on $t$, one can prove the result with $b$ possibly unbounded, only
assuming that its component $b_{i}$ are ``locally uniformly $\alpha
$-H\"{o}lder continuous'', i.e.,
\begin{equation}
\label{vai}[ b_{i}]_{\alpha,1}:=\sup_{x\neq y\in\mathbb{R}^{d}}\frac{|b_{i}
(x)-b_{i}(y)|}{(|x-y|^{\alpha}\vee|x-y|)}<+\infty, \;\;\; i=1, \ldots, d,
\end{equation}
where $a\vee b=\max(a,b)$, for $a,b\in\mathbb{R}$. Under \eqref{vai}  one can
still construct a stochastic differentiable  flow $\phi_{t}(x)$  (see Theorem
7 in \cite{FGP3}) which satisfies  properties \eqref{bound} and
\eqref{stability2} (see also Remark  \ref{inverse}) and this  allows  to
perform the same proof of Theorem  \ref{thmnoblowup}. }
\end{remark}

\section{A stability property}

The following result shows a  \textit{stability property} for the
solutions of the SPDE; such property involves  the  $weak*$ topology (or the
$\sigma(L^{\infty}(\mathbb{R}^{d}),  L^{1}(\mathbb{R}^{d}))$-topology).

\begin{proposition}
\label{stab} Assume that Hypothesis \ref{hy1} holds true for some  $\alpha
\in(0,1)$. Moreover, denote by $\phi_{t} = \phi_{0,t} $  the stochastic flow
for equation (\ref{SDE}).  Then, for any sequence $(v^{n}) \subset L^{\infty
}(\mathbb{R}^{d})$, we have:
\[
v_{n} \to v \in L^{\infty}(\mathbb{R}^{d})\;\; \text{in $weak*$ topology}
\;\;  \Longrightarrow\;\; v_{n} (\phi_{t}^{-1}(\cdot))  \to v (\phi_{t}%
^{-1}(\cdot))
\]%
\[
\;\; \text{in $weak*$ topology,}
\]
uniformly in $t \in[0,T]$, ${P}-$a.s.
\end{proposition}

\begin{proof} We prove that, $P$-a.s.,
%there exists $\Omega_0 \in {\cal F}$ with $P(\Omega_0) =1$ such that for any $\omega \in \Omega_0$,
 for any $f \in L^{1}(\mathbb{R}^{d})$ we have
\begin{equation}
\label{c7}a_{n} = \sup_{t \in[0,T]} \Big | \int_{\mathbb{R}^{d}} [v_{n}
(\phi_{t}^{-1}( y))- v (\phi_{t}^{-1}( y))] \, f(y) dy \Big | \to 0,
\end{equation}
as $n \to\infty$.

 Recall that there exists a positive constant $M$ such that
$\| v_{n}\|_{0} \le M $, $n \ge1$, and $\| v\|_{0} \le M$ and, moreover,  by the separability of $L^1(\mathbb{R}^d)$ there exists a countable dense set $D \subset C_{0}^{\infty} (\mathbb{R}^{d})$.

It is enough to check
  \eqref{c7} when $f \in D $ (with the event of probability one, possibly depending on $f$).  Indeed,  if $f \in L^{1}(\mathbb{R}^{d})$, we can consider a sequence
$(f_{N}) \subset D$ which converges  to $f$ in
$L^{1}(\mathbb{R}^{d})$ and find, ${P}-$a.s.,
%$$
%a_n \le
%\sup_{t \in [0,T]}
%\Big | \int_{\RR^d}  [v_n (\phi_t^{-1}(y))-
%v (\phi_t^{-1}(y))] \, [f(y)- f_N(y)] dy \Big |
%$$
%$$ +
%\sup_{t \in [0,T]}
%\Big | \int_{\RR^d}  [v_n (\phi_t^{-1}(y))-
%v (\phi_t^{-1}(y))] \, f_N(y) dy \Big |
%$$%
\[
a_{n} \le2 M \int_{\mathbb{R}^{d}} \, |f(y)- f_{N}(y)| dy + \sup_{t \in[0,T]}
\Big | \int_{\mathbb{R}^{d}} [v_{n} (\phi_{t}^{-1}(y))-  v (\phi_{t}^{-1}(y))]
\, f_{N}(y) dy \Big |;
\]
by the previous inequality the assertion follows easily.

%\smallskip {\it II Step. }

\smallskip To prove \eqref{c7}  for a fixed $f \in D$ we
first note that, by changing variable ($J \phi_{t} (x) $  denotes the Jacobian determinant of
$\phi_{t}$ at $x$)
\begin{equation} \label{f5}
\int_{\mathbb{R}^{d}} [v_{n}
(\phi_{t}^{-1}( y))- v (\phi_{t}^{-1}( y))] \, f(y) dy
=   \int_{K} [v (x)-  v_{n} (x)] \,  f(\phi_{ t} (x)) J
\phi_{ t} (x) dx,
\end{equation}
where we have defined the compact set $K = \pi_2 (\{(t,x) \in[0,T]\times\mathbb{R}^{d}
\; :\; \phi_{t}^{-1}(x) \in$ supp$(f) \})$, with $\pi_2(s,x)=x$, $s \in [0,T]$, $x \in \mathbb{R}^d$.

Using that, $P$-a.s., the map: $(t,x) \mapsto$ $f(\phi_{t}(x)) J \phi_{t} (x)$  is continuous
on $[0,T] \times\mathbb{R}^{d}$, we see from \eqref{f5} that
 the map: $t \mapsto \int_{\mathbb{R}^{d}} [v_{n}
(\phi_{t}^{-1}( y))- v (\phi_{t}^{-1}( y))] \, f(y) dy$
is continuous on $[0,T]$ and so, $P$-a.s.,
\begin{equation}
\label{c712}
a_n = \sup_{t \in[0,T]\cap \mathbb{Q}} \Big | \int_{\mathbb{R}^{d}} [v_{n}
(\phi_{t}^{-1}( y))- v (\phi_{t}^{-1}( y))] \, f(y) dy \Big |.
\end{equation}
By \eqref{f5} we also deduce that, $P$-a.s.,
\begin{equation}
\label{c71} \Big | \int_{\mathbb{R}^{d}} [v_{n}
(\phi_{t}^{-1}( y))- v (\phi_{t}^{-1}( y))] \, f(y) dy \Big | \to 0,\;\;\; t \in [0,T] \cap \mathbb{Q}.
\end{equation}
We finish the  proof arguing by contradiction. We consider an event $\Omega_0$ with $P(\Omega_0) =1$ such that \eqref{c712}, \eqref{c71} holds for any $\omega \in \Omega_0$ and also $(t,x) \mapsto$ $f(\phi_{}(t, \omega)(x)) J \phi (t, \omega) (x)$  is continuous
on $[0,T] \times\mathbb{R}^{d}$ for any $\omega \in \Omega_0$.

If \eqref{c7}  does not hold for some $\omega_0 \in \Omega_0$, then there exists
$\varepsilon>0$ and $(t_{n}) \subset[0,T] \cap \mathbb{Q}$ such that
\[
\Big | \int_{\mathbb{R}^{d}} [v_{n} (\phi_{t_{n}}^{-1}( y))-  v (\phi_{t_{n}%
}^{-1}(y))] \, f(y) dy \Big | > \varepsilon
\]
(we do not indicate dependence on $\omega_0$ to simplify notation;  in the sequel we always argue at $\omega_0$ fixed).
Possibly passing to a subsequence, we may assume that $t_{n} \to\hat t
\in[0,T]. $

By changing variable   we have,  for any $n \ge1$,
\[
\varepsilon< \Big | \int_{K } [v (x)-  v_{n} (x)] \, f(\phi_{t_{n}}(x)) J
\phi_{t_{n}} (x) dx \Big | \le(1) + (2),
\]
\[
(1) = \Big | \int_{K } [v (x)-  v_{n} (x)] \, [f(\phi_{t_{n}}(x)) J
\phi_{t_{n}} (x) -  f(\phi_{\hat t} (x)) J \phi_{\hat t} (x) ] dx \Big |,
\]
\[
(2) = \Big | \int_{K } [v (x)-  v_{n} (x)] \,  f(\phi_{\hat t} (x)) J
\phi_{\hat t} (x) dx \Big |.
\]
Now
\[
(1) \le 2 M \int_{K } | f(\phi_{t_{n}}(x)) J \phi_{t_{n}} (x) -  f(\phi_{\hat t
}(x)) J \phi_{\hat t} (x) | dx,
\]
which tends to 0, as $n \to\infty$, ${P}-$a.s., by the dominated convergence
theorem (indeed at $\omega_0$ fixed,  $(t,x) \mapsto$ $f(\phi(t, \omega_0) (x)) J \phi ({t}, \omega_0) (x)$  is continuous
on $[0,T] \times\mathbb{R}^{d}$).

Let us consider (2). By uniform continuity of $f(\phi_{t}( x)) J \phi_{t} (x)$ on $[0,T] \times K$ we may choose $q \in [0,T] \cap \mathbb{Q}$ such that
$$
 | f(\phi_{\hat t}(x)) J \phi_{\hat t} (x) -  f(\phi_{q
}(x)) J \phi_{q} (x) | < \frac{\epsilon } {4M \, \lambda(K)},
$$
for any $x \in K$ (here $\lambda(K)$ is the Lebesgue measure of $K$). Now, for any $n \ge 1,$
\begin{align*}
& (2) \le \Big | \int_{K } [v (x)-  v_{n} (x)] \,  [f(\phi_{\hat t} (x)) J
\phi_{\hat t} (x)- f(\phi_{q} (x)) Jf(\phi_{q} (x)) ]dx   \Big |
\\
& +   \Big | \int_{K } [v (x)-  v_{n} (x)] \,  f(\phi_{q} (x)) J
\phi_{q} (x) dx \Big |
\\
& \le \epsilon/2  + \Big | \int_{K } [v (x)-  v_{n} (x)] \,  f(\phi_{q} (x)) J
\phi_{q} (x) dx \Big |.
\end{align*}
Since $x \mapsto f(\phi_{q }(x)) J
\phi_{q} (x)$ is integrable on $\mathbb{R}^{d}$, we find that  the last term tends
tends to 0,  as $n \to\infty$.

We have found a contradiction. The proof is complete.
\end{proof}

%% fine new part preprint

\section{New uniqueness results}

\label{sec:fractional}

The aim of this section is to prove some new uniqueness results for
$L^{\infty}$ weak solutions of the SPDE obtained extending the key estimates
in fractional Sobolev spaces.

 Unlike  Theorem \ref{thmnoblowup} we will assume more conditions on $b$. On the other hand we will allow $u_0 \in L^{\infty}(\mathbb{R}^{d})$ and prove  stronger uniqueness results in the larger class of weak solutions.
   Recall that the uniqueness statement, in a class
of so regular solutions, of Theorem \ref{thmnoblowup} is rather obvious and does not
require special effort and assumptions on the drift. On the contrary, the
uniqueness claims in a class of weak solutions of Theorems \ref{thm:aux-1} and \ref{thm:aux-2} below
are quite delicate and require suitable conditions on the drift.

 The first result is the following:

\begin{theorem}
\label{thm:aux-1} Let $d \ge2$ and   $u_{0}\in
L^{\infty}\left( \mathbb{R}^{d}\right)  $.  Assume Hypothesis~\ref{hy1} and also that
$${\mathrm{div}\,}b\in L^{q} (0,T;L^{p}(\mathbb{R}^{d})) $$  for some $q>2 \ge p
> \frac{2d}{ d + 2 \alpha}$. Then there exists a unique weak $L^{\infty}%
$-solution $u$ of the Cauchy problem for the transport equation and
$u(t,x)=u_{0}({\phi}_{t}^{-1}( x)) $.
\end{theorem}

The main interest of this result is due to the fact that  we can
consider some $p$ in the critical interval  $(1,2]$ not covered by
Hypothesis 2 in \cite{FGP}; recall that this  requires that  there
exists $p \in(2,+\infty)$, such that
\begin{equation}
{\mathrm{div}\,}b\in L^{p} ([0,T]\times\mathbb{R}^{d})\label{uni},
\;\;\;\ d \ge 2.
\end{equation}
The next uniqueness result  requires  an additional hypothesis of
Sobolev regularity for $b$ (beside the usual H\"older regularity)
but allows to   avoid    {\it global} integrability assumptions on
$\mathrm{div}\,b$.
\begin{theorem}
\label{thm:aux-2}
 Assume     $u_{0}\in
L^{\infty}\left( \mathbb{R}^{d}\right)  ,$    $\mathrm{div}\, b\in
L^1_\loc([0,T]\times\RR^d)$ and
\begin{align} \label{d4}
  b \in L^1 (0,T; W_{\loc}^{\theta, 1} (\RR^d) ) \cap L^\infty (0,T; C^{\alpha} (\RR^d) )
\end{align}
with $\alpha, \theta \in (0,1)$  and $\alpha+\theta>1$. Then
 there exists a unique weak $L^{\infty}$-solution $u$ of the
Cauchy problem for the transport equations and $
u(t,x)=u_{0}({\phi}_{t}^{-1}( x)) $.
\end{theorem}

  \begin{remark} Recall that
 $b \in L^1 (0,T; W_{\loc}^{\theta, 1} (\RR^d) )$ if  $\int_0^T \| b (s, \cdot)\|_{W_{}^{\theta, 1} (\mathcal{D})} ds < \infty$,
 for any smooth bounded domain $\mathcal{D} \subset \mathbb{R}^d$.
  Since  $C^{\alpha} (\mathcal{D}) \subset W_{}^{\theta, 1} (\mathcal{D}) $, for any $\theta < \alpha$, we deduce that Hypothesis \ref{hy1}  implies  \eqref{d4} when $\alpha > 1/2$; in particular Theorem \ref{thm:aux-2}  follows from  Theorem \ref{fgpp} but only when $\alpha >1/2$.
  \end{remark}

 The proofs of both theorems follow ideas of \cite[Section 5]{FGP},
 using the results below on the commutator and on the regularity of
the Jacobian of the flow.
%Since these results are complementary the
%details of the proofs are left to the reader.
 The following
commutator estimates follows from \cite[Lemma 22]{FGP}.

\begin{corollary} \label{ci4} Assume
  $ v    \in L_{loc}^{\infty}\left(
\mathbb{R}^{d},\mathbb{R}^{d}\right) $, ${\mathrm{div}\,}v\in
L_{loc}^{1}\left(  \mathbb{R}^{d}\right)$, $g    \in
L_{loc}^{\infty}\left(  \mathbb{R}^{d}\right)$ and $\rho \in
C_r^{\infty} (  \mathbb{R}^{d})$.
\begin{itemize}
\item[(i)] If
there exists $\theta  \in (0,1)$ such that $
  v \in W_{\loc}^{\theta, 1}(
 \mathbb{R}^{d}, \mathbb{R}^{d}),
 %\;\;\; \, \rho \in
% C_{\loc}^{\alpha'}( \mathbb{R}^{d}),
$
  then %we have the  uniform bound
$$
\left|\int_{{\mathbb R}^d}\mathcal{R}_{\varepsilon}\left[
g,v\right] (  x) \rho(x) dx \right| \le
 C_r \| g\|_{L^{\infty}_{r+1}} \big( \|
\rho\|_{L^{\infty}_{r}}
 \| {\mathrm{div}\,} v
\|_{_{L_{r+1}^{1}}} +  \, [ \rho ]_{C^{1 - \theta}_{r}} \,
 [v]_{W^{\theta,1}_{r+1}} \big).
$$
\item[(ii)] If  there exists $\alpha  \in (0,1)$ such that
$
   v \in C_{\loc}^{\alpha}( \mathbb{R}^{d}, \mathbb{R}^{d}),
$
 then %we have the  uniform bound
$$
\left|\int_{{\mathbb R}^d}\mathcal{R}_{\varepsilon}\left[
g,v\right] (  x) \rho(x) dx \right| \le
 C_r \| g\|_{L^{\infty}_{r+1}} \big( \|
\rho\|_{L^{\infty}_{r}}
 \| {\mathrm{div}\,} v
\|_{_{L_{r+1}^{1}}} +  \, [ v ]_{C^{\alpha}_{r+1}} \,
 [\rho]_{W^{1- \alpha,1}_r} \big).
$$
\end{itemize}
\end{corollary}
 \begin{proof} We have
$$
  \left|
  \iint g(x') D_x \vartheta_{\eps} ({x-x'}) \,
\big( \rho(x) - \rho(x') \big)\, [v(x)-v(x')] \, dx dx' \right|
$$
$$
\le \frac{\eps^{1 - \theta}}{\eps} \, [ \rho ]_{C^{1 - \theta}_{r}}
\,
 \| g \|_{L^{\infty}_{r+1}}\,
  \frac{1}{\eps^{d}}\,
   \iint_{B(r+1)^2} |D_x \vartheta (\frac{x-x'}{\eps})|  \,
  \frac{|v(x)-v(x')|}{|x-x'|^{\theta + d}} \,
  |x-x'|^{ \theta + d}  dx dx'
$$
$$
 \le
  [ \rho ]_{C^{1 - \theta}_{r}}
\,
 \| g \|_{L^{\infty}_{r+1}}
  \| D \theta \|_{\infty}\, [v]_{W^{\theta,1}_{r+1}}
$$
The second statement has a similar proof.
\end{proof}

The previous result can be extended to the case in which commutators
are composed with  a flow.

 \begin{lemma}
 \label{dim12} Let $\phi$ be a $C^{1}$-diffeomorphism of
 $\mathbb{R}^{d}$ ($J \phi$ denotes its Jacobian). Assume
  $ v    \in L_{loc}^{\infty}\left(
\mathbb{R}^{d},\mathbb{R}^{d}\right) $, ${\mathrm{div}\,}v\in
L_{loc}^{1}\left(  \mathbb{R}^{d}\right)$, $g    \in
L_{loc}^{\infty}\left(  \mathbb{R}^{d}\right)$.

Then, for any $\rho \in C_r^{\infty}(  \mathbb{R}^{d})$ and any
$R>0$ such that $\mathrm{supp}( \rho\circ \phi^{-1}) \subseteq
B(R)$,
 we have a uniform bound of
 $\int\mathcal{R}_{\varepsilon}\left[
g,v\right] \left(  \phi\left(  x\right)  \right)  \rho\left(
x\right)  dx $ under one of the following conditions:
\begin{itemize}
\item[(i)] there exists $\theta   \in (0,1)$ such
that $
  v \in W_{\loc}^{\theta, 1}(
 \mathbb{R}^{d}, \mathbb{R}^{d})$, $ J\phi \in
C_{\loc}^{1 - \theta}( \mathbb{R}^{d}) $;

\item[(ii)] there exists $\alpha  \in (0,1)$ such that
$
  J \phi \in W_{\loc}^{1 - \alpha, 1}(
 \mathbb{R}^{d} ) $, $ v \in
C_{\loc}^{\alpha}( \mathbb{R}^{d}, \mathbb{R}^{d}) $.
\end{itemize}
Moreover, under one of the previous conditions, we also have
\[
\lim_{\varepsilon\rightarrow0}\int\mathcal{R}_{\varepsilon}\left[
g,v\right] \left(  \phi\left(  x\right)  \right)  \rho\left(
x\right)  dx=0.
\]
\end{lemma}
\begin{proof}
 By a change of variables $ \int\mathcal{R}_{\varepsilon}[
g,v] (  \phi(  x) )  \rho( x) dx=\int\mathcal{R}_{\varepsilon}[ g,v]
(  y)  \rho_{\phi}(  y) dx $ where the function
 $ \rho_{\phi}(  y)  =\rho(  \phi^{-1}( y) )
 J\phi^{-1}(  y)
 $
 has the support strictly contained in the ball
of radius $R$. Clearly, $
 \|  \rho_{\phi}\| _{L_{R}^{\infty}}\leq \| \rho\|
_{L_{r}^{\infty}}\| J\phi^{-1}\| _{L_{R}^{\infty}}. $ To prove the
result, we  have to check that Corollary~\ref{ci4} can be applied
with $\rho_{\phi}$ instead of $\rho$.

(i) To apply Corollary~\ref{ci4}~(i), we need to check that
 $ \rho_{\phi} \in C^{1 -\theta}_{loc}$. This follows since
$$
[\rho_{\phi} ]_{ C^{1 -\theta}_{R}} \le
 \| J\phi^{-1}\|
_{L_{R}^{\infty}} \, [\rho ({\phi}^{-1} (\cdot) ) ]_{ C^{1
-\theta}_{R}} \, + \, \| \rho \|_{L_{r}^{\infty}}
 [J\phi^{-1}]_{C^{1 -\theta}_{R}}
$$
$$
\le  \| D\phi^{-1}\| _{L_{R}^{\infty}} \,  \| D\rho\|
_{L_{r}^{\infty}} \, [ D\phi^{-1}]_{C^{1 -\theta}_{R}}
 \, + \, \| \rho \|_{L_{r}^{\infty}}
 [D\phi^{-1}]_{C^{1 -\theta}_{R}}.
$$
and the bound follows.

\medskip (ii) To apply Corollary~\ref{ci4}~(ii), we
  need to check that
 $ \rho_{\phi} \in W^{1- \alpha,1}_{loc}$: first
$$ [\rho_{\phi} ]_{W^{1 - \alpha, 1}_R }
\le  \| J\phi^{-1}\| _{L_{R}^{\infty}} [\rho \circ \phi^{-1}]_{W^{1-
\alpha,1}_R}
 \, + \,
 [J \phi^{-1} ]_{W^{1- \alpha,1}_R} \, \| \rho\|
_{L_{r}^{\infty}}
$$
and since
$$
[\rho \circ \phi^{-1}]_{W^{1- \alpha,1}_R} \le \|D (\rho \circ
\phi^{-1})\|_{L^1_R} \le \|D\rho\|_{L^1_r} \|D
\phi^{-1}\|_{L^\infty_R}
$$
we find
$$
[\rho_{\phi} ]_{W^{1 - \alpha, 1}_R } \, \le \,  C_R
\|D\rho\|_{L^1_r} \|D \phi^{-1}\|_{L^\infty_R} \| J\phi^{-1}\|
_{L_{R}^{\infty}} + [J \phi^{-1} ]_{W^{1- \alpha,1}_R} \, \| \rho\|
_{L_{r}^{\infty}}
$$
and the bound follows.
\end{proof}

Finally the next theorem extends the analysis of the Jacobian of the
flow presented in Section 2 and links the regularity condition on
$J\phi$ required in  Lemma~\ref{dim12} (ii) to the assumption on the
divergence of $b$ stated in Theorem~\ref{thm:aux-1}.

\begin{theorem}
\label{iac11} Let $d\ge 2$. Assume  Hypothesis~\ref{hy1}  and  the
existence of  $p\in (\frac{2d}{ d + 2 \alpha},2]$ and $q>2$ such
that $
 {\mathrm{div}\,}b\in L^{q} (0,T;L^{p}(\mathbb{R}^{d})) .
$ Then,
 for any
 $r>0$,
 $
  J\phi \in
 L^{p}( 0,T;W_{r}^{1 - \alpha, \, p} )$, $P$-a.s.
\end{theorem}
\begin{proof}  In the sequel we assume $\sigma =1$
 to simplify notation.

 The first part of the proof is similar to
 the one of \cite[Theorem 11]{FGP}.
  Indeed Step~1 can be carried on thanks to the  chain rule for fractional Sobolev
 spaces:
 if $f:\mathbb{R}^{d}\rightarrow \mathbb{%
R}$ is a continuous function, of class $W_{loc}^{1-\alpha,p}(\mathbb{R}^{d})$ and $%
g:\mathbb{R^d}\rightarrow \mathbb{R}$ is a $C^{\infty }$ function,
then $ g\circ f\in W_{loc}^{1-\alpha,p}(\mathbb{R}^{d})$  and
\begin{equation*}
 [ (g\circ f)]^p_{W_{r}^{1- \alpha, p}}
 \leq \left( \sup_{x\in B(r)}\left|
g^{\prime }(f(x))\right| \right) ^{p} [  f]^p_{W_{r}^{1- \alpha,
p}},
\end{equation*}
for every $r>0$. The modification of Step~2 does not pose any
problem, so we only consider the last steps of the proof.

\medskip \textbf{Step 3.} To prove the assertion it is enough
 to check that the family $\left(
\psi _{\varepsilon }\right) _{\varepsilon >0}$ is bounded in
 $L^{p}(\Omega \times (0,T);W_{r}^{1-\alpha,p})$.

Indeed, once we have proved this fact, we can extract from
 the previous
sequence
 $\psi _{\varepsilon _{n}}$ a subsequence which converges
  weakly in $L^{p}(\Omega
\times (0,T);W_{r}^{1-\alpha,p})$ to some $\gamma $. This in
particular implies that such subsequence  converges weakly in
$L^{p}(\Omega \times (0,T),L_{r}^{p})$ to  $\gamma $ so we must have
that $\gamma = J \phi$.

We introduce the following Cauchy problem, for $\eps \ge 0$,
\begin{equation}
\left\{ \begin{aligned} \frac{\partial F^{\eps}}{\partial t}+
\frac{1}{2} \Delta F^{\eps} + D F^{\eps} \cdot b^{\eps}
={\mathrm{div}\, }b^{\eps}, \;\;
\; t \in [0,T[ \\ F^{\eps}(T,x)=0,\;\;\; x \in \mathbb{R} ^d. \end{aligned}%
\right.   \label{equa1}
\end{equation}
  This  problem
 has a unique solution $F^{\eps}$ in the space
 $L^{q}(0,T;W^{2,p}(\mathbb{ R}^{d})$. Moreover,
there exists a positive constant $C=C(p,q,
 d,T,\Vert b\Vert _{\infty
})$ such that
\begin{equation}
\Vert F^{\eps}\Vert _{L^{q}(0,T;W^{2,p}(\mathbb{ R}^{d})) }\leq
C\Vert
 \mathrm{div \,}b \Vert _{L^{q}(0,T;L^{p}(\mathbb{
R}^{d}))},  \label{bond5}
\end{equation}
for any $\eps \ge 0$.  This result  can be proved by using
\cite[Theorem 1.2]{Kr1}  and repeating the
 argument  of the proof
 in \cite[Theorem 10.3]{Kry-Ro}. This argument  works without
difficulties in the present case in which $b$ (and so $b^{\eps}$) is
globally bounded and $\mathrm{div \,}b \in
L^{q}(0,T;L^{p}(\mathbb{%
R}^{d}))$ with  $p, q \in (1, + \infty)$.

From the previous result we can also deduce, since we are assuming
$q>2$, that  $F^{\eps} \in C( [ 0,T] ;W^{1,p}( \mathbb{R}^{d}) )$,
for any $\eps \ge 0$, and moreover there exists a positive constant
$C$ $=C(p,q$ $
 d,T, \Vert b \Vert _{\infty
})$ such that
\begin{equation} \label{df}
\sup_{t\in\left[ 0,T\right] }\left\|   F^{\eps}(  t,\cdot) \right\|
_{W^{1,p}\left( \mathbb{R}^{d}\right)  } \leq C \Vert
 \mathrm{div \,}b \Vert _{L^{q}(0,T;L^{p}(\mathbb{
R}^{d}))}.
\end{equation}
We only give a  sketch of proof of \eqref{df}.
 Define $u^{\eps} (t,x) = F^{\eps} (T-t, x)$; we
have the explicit formula
$$
 u^{\eps}(t,x) = \int_0^t P_{t-s}g^{\eps}(s, \cdot)(x) ds,
$$
where $(P_t) $ is the heat semigroup and $g^{\eps}(t,x)
 = D u^{\eps} (t,x) \cdot b^{\eps} (T-t,x)
 - {\mathrm{div}\, }b^{\eps} (T-t,x)$. We get, since $q>2$
  and $q' =\frac{q}{q-1}<2$,
\begin{equation*}
\begin{split}
\| D_x u^{\eps}(t, \cdot)\|_{L^p}  & \le  c\int_0^t
 \frac{1}{(t-s)^{1/2}} \| g^{\eps}(s, \cdot)\|_{L^p} ds
\\ & \le C
 \Big( \int_0^T
 \frac{1}{s^{q'/2}}  ds
  \Big)^{1/q'} \, \Big( \int_0^T
   \| \mathrm{div}\,  b (s, \cdot)\|_{L^p}^{q} ds \Big)^{1/q}
\end{split}
\end{equation*}
  and so \eqref{df} holds.
  Using It\^{o} formula we find (remark that
 $F^{\eps}(t, \cdot) \in C^{2}_b(\RR^d)$)
\begin{equation} \label{ci51}
F^{\varepsilon }\left( t,{\phi }_{t}^{\varepsilon }\left( x\right)
\right)
-F^{\varepsilon }\left( 0,x\right) -\int_{0}^{t}DF^{\varepsilon }\left( s,{%
\phi }_{s}^{\varepsilon }\left( x\right) \right) \cdot dW_{s}=
\int_{0}^{t}%
\mathrm{div}b^{\varepsilon }\left( s,{\phi }_{s}^{\varepsilon
}\left( x\right) \right) ds = \psi _{\varepsilon }(t,x ).
\end{equation}
Since we already know that $\left( \psi _{\varepsilon }\right)
_{\varepsilon >0}$ is bounded in $L^{p}(\Omega \times
(0,T),L_{r}^{p})$ and since $p\le 2$, to verify that $\left( \psi
_{\varepsilon }\right) _{\varepsilon >0}$ is bounded in
$L^{p}(\Omega \times (0,T);W_{r}^{1-\alpha,p})$, it is enough to
prove that $ E \int_0^T [\psi _{\varepsilon }(t,
\cdot)]_{W_{r}^{1-\alpha,2}}^2 dt \le C, $ for any $\eps >0$. We
give details only for the most difficult term
$\int_{0}^{t}D^{}F^{\varepsilon }(s,{\phi }
 _{s}^{\eps}(x))  dW_{s}$ in \eqref{ci51}. The $F(0,x)$ term can be controlled using~\eqref{df} and the others are of easier estimation.
% The other terms are easier to estimate.
 We show  that  there exists  a constant $C>0$ (independent on  $\eps$)
 such that
\begin{equation}
E \int_{0}^{T} dt \left [ \int_{0}^{t}D^{}F^{\varepsilon }\left(
s,{\phi }_{s}^{\eps}\left( \cdot\right) \right) dW_{s}\right
]_{W_{r}^{1-\alpha,2}}^{2} \, \leq C \label{cia2}
\end{equation}
 We have
$$
E\left[
\int_{0}^{T}dt\int_{B(r)}\int_{B(r)}\frac{|\int_{0}^{t}(DF^{\varepsilon
}\left( s,{\phi}_s^{\varepsilon} \left( x\right) \right)
-DF^{\varepsilon }\left( s, {\phi}_s^{\varepsilon}\left( x^{\prime
}\right) \right) )dW_{s}|^{2}}{|x-x^{\prime }|^{(1 - \alpha)2 + d
}}dx\,dx^{\prime }\right]
$$
$$
= \int_{0}^{T}    \int_{B(r)}\int_{B(r)} E
\int_{0}^{t}\frac{|DF^{\varepsilon }\left(
s,{\phi}_s^{\varepsilon}\left( x\right) \right) -DF^{\varepsilon
}\left( s,{\phi}_s^{\varepsilon}\left( x^{\prime }\right) \right)
|^{2}}{|x-x^{\prime }|^{(1 - \alpha)2 + d } }ds \,  dx\,dx^{\prime }
,
$$
\begin{equation*}
=  E \int_{0}^{T}dt  \int_{0}^{t}ds
\int_{B(r)}\int_{B(r)}\frac{|DF^{\varepsilon }\left(
s,{\phi}_s^{\varepsilon}\left( x\right) \right) -DF^{\varepsilon
}\left( s,{\phi}_s^{\varepsilon}\left( x^{\prime }\right) \right)
|^{2}}{|x-x^{\prime }|^{(1 - \alpha)2 + d } }dx\,dx^{\prime }
\end{equation*}
\begin{equation*}
\leq TE\left[ \int_{0}^{T}ds\int_{B(r)}\int_{B(r)}\frac{|DF^{\varepsilon }\left( s,%
{\phi}_s^{\varepsilon}\left( x\right) \right) -DF^{\varepsilon
}\left( s, {\phi}^{\varepsilon}_{s}\left( x^{\prime }\right) \right)
|^{2}}{|x-x^{\prime }|^{(1 - \alpha)2 + d} }dx\,dx^{\prime }\right]
,
\end{equation*}
\begin{equation*}
\leq TE\int_{0}^{T} [ DF^{\varepsilon
}(s,{\phi}^{\varepsilon}_{s}(\cdot ))]_{W_{r}^{1 - \alpha, 2}}^{2}
\, ds
\end{equation*}
By the Sobolev embedding the $W^{1 - \alpha,2}_r$-seminorm  can be
controlled by the norm in $W^{1,p}_r$ if
$$
 1 - \frac{d}{p} \ge (1-\alpha) - \frac{d}{2}.
$$
This holds if $p \ge \frac{2d}{ d + 2 \alpha}$.
 Then we consider $p_{1} $ such that
$ p>{p_{1}}> \frac{2d}{ d + 2 \alpha} $
 and show that
\begin{equation} \label{f7}
E\int_{0}^{T}\Vert DF^{\varepsilon
}(s,{\phi}^{\varepsilon}_{s}(\cdot ))\Vert _{W_{r}^{1,p_{1}}}^{2}ds
\le C < \infty,
\end{equation}
where $C$ is independent on $\eps$.

\medskip \noindent \textbf{Step 4.}
 To obtain \eqref{f7} we estimate
\begin{equation*}
E\int_{0}^{T}ds\Big (\int_{B(r)}|D^{2}F^{\varepsilon }\left( s,{\phi}^{\varepsilon}_{s}\left( x\right) \right) D{\phi}^{\varepsilon}_{s}\left( x\right) |^{p_{1}}dx\Big )^{%
\frac{2}{p_{1}}}
\end{equation*}
A similar term has been already estimated in the proof of Theorem 11
in \cite{FGP}. Since
\begin{equation*}
\int_{B(r)}\left( \int_{0}^{T}E\left[ \left|
D{\phi}^{\varepsilon}_{s}\left( x\right) \right| ^{r}\right]
ds\right) ^{\gamma }dx<\infty ,
\end{equation*}
for every $r,\gamma \geq 1$ (see (\ref{bound})), by the H\"{o}lder
inequality, it is sufficient to prove that
\begin{equation*}
\int_{0}^{T}E\left[ \left( \int_{B(r)}\left| D^{2}F^{\varepsilon
}\left( s,{\phi}^{\varepsilon}_{s}
\left( x\right) \right) \right| ^{p^{}}dx\right) ^{\frac{2}{%
p^{ }}}\right] dt \le C <\infty.
\end{equation*}
  We have
\begin{align*}
& \int_{0}^{T}E\left[ \left( \int_{B(r)}\left| D^{2}F^{\varepsilon }
\left( s,
{\phi}^{\varepsilon}_{s}\left( x\right) \right) \right|^{p^{}}dx\right) ^{\frac{2}{%
p^{ }}}\right] dt \\
& =E\left[ \int_{0}^{T}ds\left(
\int_{{\phi}^{\varepsilon}_{s}(B(r))}\left| D^{2}F^{\varepsilon
}\left( s,{y}\right) \right|
^{p^{}}J({\phi}^{\varepsilon}_{s})^{-1}(y)dy\right) ^{\frac{2}{p^{ }}}\right]  \\
& \leq \sup_{s \in [0,T] , \, y \in \RR^d}
   E[ J (\phi_s^{\eps})^{-1}\, (y)]^{2/p}
\int_{0}^{T}\Big (\int_{\mathbb{R}^d}\left| D^{2}F^{\varepsilon
}\left( s,{y}\right) \right| ^{p}dy\Big )^{\frac{2}{p}}\, \le C <
\infty,
\end{align*}
where, using  the results of \cite[Section 3]{FGP} and the bound
\eqref{bond5}, $C$ is independent on $\eps>0.$ The proof  is
complete.
\end{proof}

%%%%%%%%%%%%%%%%%%%%%%%%%%%%%%%%%%%%%%%%%%%%%%%%%%%%%%%%%%%%%%%%%%%%%%%%%%%%%%%%%%%%%%%%%%
%%%%%%%%%%%%%%%%%%%%%%%%%%%%%%%%%%%%%%%%%%%%%%%%%%%%%%%%%%%%%%%%%%%%%%%%%%%%%%%%%%%%%%%%%%
%%%%%%%%%%%%%%%%%%%%%%%%%%%%%%%%%%%%%%%%%%%%%%%%%%%%%%%%%%%%%%%%%%%%%%%%%%%%%%%%%%%%%%%%%%
%%%%%%%%%%%%%%%%%%%%%%%%%%%%%%%%%%%%%%%%%%%%%%%%%%%%%%%%%%%%%%%%%%%%%%%%%%%%%%%%%%%%%%%%%%

%%%%%%%%%%%%%%%%%%%%%%%%%%%%%%%%%%%%%%%%%%%%%%%%%%%%%%%%%%%%%%%%%%%%%%%%%%%%%%%%%%%%%%%%%%
%%%%%%%%%%%%%%%%%%%%%%%%%%%%%%%%%%%%%%%%%%%%%%%%%%%%%%%%%%%%%%%%%%%%%%%%%%%%%%%%%%%%%%%%%%

\end{document}